\documentclass[11pt,twoside]{article}

\usepackage[utf8]{inputenc}
\usepackage{hyperref}
\usepackage{amsmath}
\usepackage{amssymb}
\usepackage{times}
\usepackage{anyfontsize}
\usepackage{changepage}
\usepackage{booktabs}
\usepackage{stmaryrd}
\usepackage{url}
\usepackage{longtable}
\usepackage{titlesec}
\usepackage{amsthm}
\usepackage{fancyhdr}
\usepackage[letterpaper]{geometry}
\usepackage{lineno}
	
\newcommand{\RR}{\mathbb{R}}

\newcommand{\TAB}{\hspace{0.50cm}}
\newcommand{\IFF}{\Leftrightarrow}
\newcommand{\IMP}{\Rightarrow}

\newcommand{\cm}{\mathfrak{c}}
\newcommand{\mc}{\mathcal}
\newcommand{\mf}{\mathfrak}
\newcommand{\bez}{\backslash}
\newcommand{\se}{\subseteq}

\newcommand{\lc}{\lceil}
\newcommand{\rc}{\rceil}

\newtheorem{mydef}{Definition}[section]
\newtheorem{myth}{Theorem}[section]

\newtheorem{mylemma}{Lemma}[section]

\newtheorem{myremark}{Remark}[section]

\titleformat*{\section}{\filcenter}
\titlelabel{\thetitle. }

\title{Zacny tytuł}
\date{}
\author{Marcin Michalski}

\begin{document}
\thispagestyle{empty}
\begin{flushright}
\end{flushright}
\hfill
\begin{flushleft}
Marcin MICHALSKI\footnote{Wrocław University of Science and Technology, Wybrzeże Wyspiańskiego 27, 50-370 Wrocław. The author was supported by grant S50129/K1102 (0401/0230/15), Faculty of Fundamental Problem of Technology.}
\end{flushleft}
\vspace{0.5cm}
\begin{center}
\textbf{{\fontsize{13}{13}\selectfont ON SOME RELATIONS BETWEEN PROPERTIES OF INVARIANT $\sigma$-IDEALS IN POLISH SPACES}}
\end{center}
\vspace{0.5cm}
\begin{adjustwidth}{0.5cm}{0pt}
\footnotesize In this paper we shall consider a couple of properties of $\sigma$-ideals and study relations between them. Namely we will prove that $\cm$-cc $\sigma$-ideals are tall and that the Weaker Smital Property implies that every Borel $\mc{I}$-positive set contains a witness for non($\mc{I}$) as well, as satisfying ccc and Fubini Property. We give also a characterization of nonmeasurability of  $\mc{I}$-Luzin sets and prove that the ideal $[\RR]^{\leq\omega}$ does not posses the Fubini Property using some interesting lemma about perfect sets.
\end{adjustwidth}
\vspace{0.5cm}
\section{INTRODUCTION, DEFINITIONS AND NOTATION}

\TAB In  \cite{MSZ} the authors achieved a somehow unexpected result that so called Weaker Smital Property implies that  $\mc{I}$-Luzin sets are  $\mc{I}$-nonmeasurable. This fact was a motivation to seek more precise characterization of $\mc{I}$-nonmeasurability of $\mc{I}$-Luzin sets and establish relations between other properties of $\sigma$-ideals.
We will use standard set theoretic notation as in \cite{JECH}. We denote the real line by $\RR$ and for any $n\in \omega$ (which denotes the set of natural numbers) and a set $A$ we define the Cartesian product by $A^n$ . By  $\alpha, \beta, \gamma, \kappa$, $…$ we denote ordinal numbers with the special case of $\cm$, which stands for the continnum - the cardinality of  $\RR$. A cardinality of any set $A$ we will denote by $|A|$. If for a set $A$ we have $|A|\leq \omega$, then we say that $A$ is countable. Otherwise we shall say it is uncountable. By $CH$ we will denote the continuum hypothesis - a statement that there is no cardinal $\kappa$ such that $\omega<\kappa<\cm$. Let $X$ be a Polish space. By $\mc{I}\se P(X)$ we will denote a $\sigma$-ideal, i.e. a family of sets which is closed under taking subsets and countable unions. For any sets $A, B$ we define their algebraic sum as follows:
\begin{linenomath*}
$$
A+B=\{a+b: a\in A\ \ \&\ \ b\in B\}
$$
\end{linenomath*}
Throughout the paper we will consider only $\sigma$-ideals that are nontrivial ($X\notin\mc{I}$), contain all singletons, and are invariant, i.e. for any $x\in X$ and $A\in\mc{I}$ we have $-A+x\in\mc{I}$ ($-A=\{-a: a\in A\}$), and possess a Borel base - for every $A\in\mc{I}$ there exists a Borel set $B\in\mc{I}$ such that $A\se B$. For such ideals we will consider one of the cardinal coefficients from Cichoń's Diagram:
\begin{linenomath*}
$$
non(\mc{I})=min\{|A|: A\notin\mc{I}\ \ \&\ \ A\se X\}.
$$
\end{linenomath*}
We call a Borel set $B$ $\mc{I}$-positive, if $B\notin\mc{I}$ . A set $B$ such that $B^c\in \mc{I}$ will be called $\mc{I}$-residual.
For a given $\sigma$-ideal $\mc{I}$ we will consider the following properties.

\begin{mydef}
We say that $\mc{I}$
\begin{itemize}
\item is $ccc$, if every family of pairwise disjoint $\mc{I}$-positive Borel sets is countable;
\item is $\kappa$-$cc$, if every family of pairwise disjoint $\mc{I}$-positive Borel sets has a cardinality smaller than $\kappa$;
\item is tall, if each $\mc{I}$-positive Borel set contains a perfect subset from $\mc{I}$;
\item has the Weaker Smital Property (see \cite{BFN}), if there is a countable and dense set $D$ such that for every Borel $\mc{I}$-positive set $B$ the sum $D+B$ is $\mc{I}$-residual (we say then, that $D$ witnesses the Weaker Smital Property).
\end{itemize}
\end{mydef}
To illustrate these properties let us observe that the two classic examples of  $\sigma$-ideals on the real line - the family of meager sets $\mc{M}$ and the family o Lebesgue measure zero sets $\mc{N}$ have all of the above properties. Possessing the Weaker Smital Property follows by the Picard and the Steinhaus Theorem respectively and being tall is assured by the existence in every infinite Borel set some type of a Cantor set.
\begin{mydef}
We say, that a set $A$
\begin{itemize}
\item witnesses $non(\mc{I})$, if $A\notin\mc{I}$ and $|A|=non(\mc{I})$;
\item is $\mc{I}$-nonmeasurable, if $A\notin\sigma(\mc{B}\cup\mc{I})$- a  $\sigma$-field generated by the family of Borel sets and $\mc{I}$;
\item is an $\mc{I}$-Luzin set, if for every set $I\in\mc{I}$ we have $|A\cap I|<|A|$.
\end{itemize}
\end{mydef}
$M$-Luzin sets are usually called in the literature generalized Luzin sets and $N$-Luzin sets generalized Sierpiński sets. It is easy to check  that Luzin sets do not have the Baire Property and Sierpiński sets are nonmeasurable. On the other hand, if  $\mc{I}=[\RR]^{\leq\omega}$, $\mc{I}$-Luzin sets can be $\mc{I}$-measurable (for instance the real line is an $[\RR]^{\leq\omega}$-Luzin set).
Also, we will be concerned with so called product ideals. For $\sigma$-ideals $\mc{I}$ and $\mc{J}$ of spaces $X$ and $Y$ respectively we define a Fubini product of these $\sigma$-ideals as follows:
\begin{linenomath*}
$$
A\in \mc{I}\otimes \mc{J}\IFF(\exists B\in\mc{B})(A\subseteq B\ \ \&\ \ \{x\in X: B_x\notin\mc{J}\}\in\mc{I}),
$$
\end{linenomath*}
where $B_x=\{y: (x,y)\in B\}$ is a vertical slice of $B$ in point $x\in X$. We say that a pair of $\sigma$-ideals $(\mc{I},\mc{J})$ has the Fubini property, if for every Borel set $B\se X\times Y$ we have $\{x\in X: B_x\notin \mc{J}\}\in\mc{I}\IMP \{y\in Y: B^y\notin \mc{I}\}\in\mc{J}$. We will say that $\mc{I}$ satisfies the Fubini Property, if $(\mc{I},\mc{I})$ has it.

Now we may proceed with the results.

\section{ON PROPERTIES OF $\sigma$-IDEALS}

We will be concerned with the following properties: ccc, being tall, the Weaker Smital Property and possessing a witness for $non(\mc{I})$ in every Borel $\mc{I}$-positive set.
We will start with a nice characterization of  $\mc{I}$-nonmeasurability of  $\mc{I}$-Luzin sets by tallness.
\begin{myth}
Every $\mc{I}$-Luzin set is $\mc{I}$-nonmeasurable if and only if $\mc{I}$ is tall.
\end{myth}
\begin{proof}
$"\Leftarrow"$: Suppose that there is an $\mc{I}$-measurable $\mc{I}$-Luzin set $L$. Then there exist $B\in\mc{B}$ and $I\in\mc{I}$ such that $L=B\Delta I$. $B\notin\mc{I}$ so there exists perfect set $P\se B$ from $\mc{I}$ for which $|L\cap P|=\mf{c}$.
\\
$"\Rightarrow"$: Suppose the opposite, i.e. there is a Borel set $B\notin\mc{I}$ containing no perfect set from $\mc{I}$. We shall prove that such a set is an $\mc{I}$-Luzin set. If it is not then there exists $I\in\mc{I}$ (we may assume that it is Borel) such that $|B\cap I|=\mf{c}$. $B\cap I$ is a Borel set from $\mc{I}$, hence it contains a perfect subset from $\mc{I}$,  a contradiction. So indeed $B$ is an $\mc{I}$-Luzin set and, by the assumption, it is $\mc{I}$-nonmeasurable and Borel. The contradiction completes the proof.
\end{proof}

It seems that being tall is rather common property, which is reflected by the following theorem.
\begin{myth}
If $\mc{I}$ is ccc, then $\mc{I}$ is tall.
\end{myth}
\begin{proof}
Suppose that $\mc{I}$ is ccc, but not tall. Then there is a Borel $\mc{I}$-positive set $B$ such that every its perfect subset is not in $\mc{I}$. Since every perfect set can be divided into $\cm$ perfect subsets, we may find $\cm$ many pairwise disjoint Borel sets outside the $\mc{I}$. Contradiction completes the proof.
\end{proof}

\begin{myremark}
Actually we have shown that if $\mc{I}$ is $\mf{c}$-cc, then $\mc{I}$ is tall.
\end{myremark}

\begin{myth}
The Weaker Smital Property implies that every Borel $\mc{I}$-positive set contains a witness for $non(\mc{I})$.
\end{myth}
\begin{proof}
Let $B$ be an $\mc{I}$-positive set and $A\notin \mc{I}$. Let $D$ be a witness for the Weaker Smital Property. Then $B+D$ is $\mc{I}$-residual, so $A$ is almost entirely (besides a set from $\mc{I}$) contained in $B+D$ (otherwise it would be a subset of $(B+D)^c$ which lies in $\mc{I}$ be definition). Since
\begin{linenomath*}
$$
(B+D)\cap A=\bigcup_{d\in D}\big((B+d)\cap A\big)
$$
\end{linenomath*}
and $D$ is countable, not for all $d\in D$ $(B+d)\cap A$ is in $\mc{I}$, so there is $d\in D$ such that $(B+d) \cap A$ is outside of $\mc{I}$. Now, if $A$ has a cardinality of $non(\mc{I})$, then the part of $A-d$ that  is in $B$ cannot be of any smaller cardinality (otherwise $(A-d)\cap B$ would be in $\mc{I}$), so $(A-d)$ is the desired set.
\end{proof}

\begin{myth}
If $\mc{I}$ is ccc and satisfies the Fubini Property, then every Borel $\mc{I}$-positive set has a witness for $non(\mc{I})$.
\end{myth}
\begin{proof}
Let $B$ be a Borel $\mc{I}$-positive set and let $A\notin\mc{I}$ has the cardinality of $non(\mc{I})$. Let us assume that $A\cap B\in\mc{I}$ (otherwise we are done). We shall find $x\in X$ such that $(x+B)\bez B\notin\mc{I}$. Let us consider a set $C\subset X\times X$ defined by the formula
\begin{linenomath*}
$$
C=\{(x,y): y\in(x+B)\bez B\}.
$$
\end{linenomath*}
Now, if we fix $y\in X$, then $C^y=\{x: y\notin B\,\, \&\,\, y\in (x+B)\}$ and $y\in (x+B)$ is equivalent to $x\in (y-B)$. Since we may  choose $y\notin B$ arbitrarily, $B^c$ is $\mc{I}$-positive, and the set $y-B$ is $\mc{I}$-positive, we have that $C^y$ is also $\mc{I}$-positive. Eventually, it means that $\{y: C^y\notin \mc{I}\}\notin\mc{I}$, so by the Fubini Property also $\{x: C_x\notin \mc{I}\}\notin\mc{I}$, so there exists $x$ such that
\begin{linenomath*}
$$
C_x=\{y: y\in (x+B)\bez B\}=(x+B)\bez B\notin\mc{I},
$$
\end{linenomath*}
what is excatly what we were looking for. If $((x+B)\bez B) \cap A\notin\mc{I}$, then $A-x$ is the set. If not, we repeat the procedure to find $x_1$ such that $(x_1+B)\bez ((x+B)\cup B)\notin\mc{I}$. We proceed as long as necessary, with a guarantee, that it stops at some countable step $\alpha$ due to the ccc property. Formally, let us define a sequence of pairwise disjoint sets
\begin{linenomath*}
\begin{align*}
B_0&=B,
\\
B_{\alpha}&=(x_{\alpha}+B)\bez \bigcup_{\beta<\alpha}B_\beta,
\end{align*}
\end{linenomath*}
where $x_\alpha$ is such that $B_{\alpha}$ is $\mc{I}$-positive. By the ccc the length of this sequence must be countable. Let $\kappa$ be the least ordinal such that $B_\kappa\cap A\notin\mc{I}$ (if for each $\alpha$ the set $B_\alpha\cap A$ was in $\mc{I}$ then $A$ also would be in $\mc{I}$). Then $B\cap (A-x_\alpha)$ is the set.
\end{proof}

\section{$[\RR]^{\leq\omega}$ DOES NOT HAVE THE FUBINI PROPERTY}

Let us start with a helpfull lemma (see also \cite{MSZ}, Lemma 2.1).
\begin{mylemma}\label{perfectmigacz}
Every perfect set $P\se \RR^n$ contains a perfect subset $Q$ such that $|Q\cap (Q+x)|\leq 1$ for each $x\neq 0$.
\end{mylemma}
\begin{proof}
Let $P$ be a perfect set. Let $Q_0=\overline{B}(x_1^0,1)\cap P$ for arbitrarily chosen $x_1^0\in P$. Since each point in $P$ is a condensation point, $Q_0$ is perfect. Let us assume that we are at the step $n+1$ and we have a sequence of $n$ nested closed sets
\begin{linenomath*}
$$
Q_n\se Q_{n-1}\se ... \se Q_0,
$$
\end{linenomath*}
such that for each $i\leq n$ $Q_i$ consists of $2^i$ pairwise disjoint closed balls $\overline{B}(x^i_k,r_i)$, $0<k\leq 2^i$, intersected with $P$. Let us choose $x^{n+1}_1, x^{n+1}_2 \in \overline{B}(x^n_1,r_n)\cap P$ and for each $2<k\leq 2^{n+1}$ pick $\epsilon_k>0$ such that
\begin{linenomath*}
$$
B(x^{n}_{\lc \frac{k}{2} \rc}, r_n)\cap P\bez \bigcup_{0<i, j, l<k}\overline{B}(x^{n+1}_i+(x^{n+1}_{l}-x^{n+1}_j), 3\epsilon_k)
$$
\end{linenomath*}
is uncountable. Pick $x^{n+1}_k$ from this set. Now let us set $r_{n+1}>0$ smaller than every $\epsilon_k$ and small enough so that balls $\overline{B}(x^{n+1}_k,r_{n+1})$, $0<k\leq2^{n+1}$, are pairwise disjoint, contained in $Q_n$ and intersect $P$ on a nonempty set. Finally, set 
\begin{linenomath*}
$$
Q_{n+1}=\bigcup_{0<k\leq 2^{n+1}}\overline{B}(x^{n+1}_k,r_{n+1})\cap P.
$$
\end{linenomath*}
This finishes the construction and $Q=\bigcap_{n\in\omega}Q_n$ is the set.
\end{proof}

\begin{myth}
The ideal $[\RR]^{\leq\omega}$ does not posses the Fubini Property.
\begin{proof}
Let $Q$ be a perfect subset of $\RR$ with the property as in Lemma \ref{perfectmigacz}. Set
\begin{linenomath*}
$$
B=\{(x,y)\in\RR^2: x\in (-Q+y), y\in Q\}\bez \{0\}\times \RR
$$
\end{linenomath*}
Clearly, $B$ is Borel ($B=f^{-1}[Q\times Q]$, where $f(x,y)=(-x+y, y)$), each vertical slice is at most $1$-point and uncountably many horizontal slices are uncountable. 
\end{proof}

\end{myth}


\begin{thebibliography}{123}

\footnotesize\bibitem{MSZ} MICHALSKI M., ŻEBERSKI SZ., Some properties of $\mc{I}$-Luzin sets, Topology and its Applications, Elsevier, vol. 189, 122-135, 2015.
\bibitem{JECH} \footnotesize JECH T., Set Theory, millennium edition, Springer Monographs in Mathematics, Springer-Verlag, 2003.
\bibitem{BFN} BARTOSZEWICZ A., FILIPCZAK M., NATKANIEC T., On Smital properties, Topology and its Applications, Elsevier, vol. 158, 2066-2075, 2011.


\end{thebibliography}
\end{document}